\documentclass[11pt]{article}

\usepackage{stmaryrd}
\usepackage{amsmath}
\usepackage{amsfonts}
\usepackage{amssymb}
\usepackage{amsthm}
\usepackage{multirow}
\usepackage{graphicx}
\usepackage{cite}
\usepackage{empheq}
\usepackage{caption}
\usepackage{subcaption}
\usepackage{url}
\usepackage{enumitem}

\theoremstyle{plain}
\newtheorem{definition}{Definition}
\newtheorem{lemma}{Lemma}
\newtheorem{theorem}{Theorem}
\newtheorem{corollary}{Corollary}
\newtheorem{proposition}{Proposition}
\newtheorem*{ruleofthumb}{Rule of Thumb}

\theoremstyle{definition}
\newtheorem{example}{Example}

\theoremstyle{remark}
\newtheorem{remark}{Remark}

\newcommand{\C}{\mathbb C}
\newcommand{\R}{\mathbb R}
\newcommand{\Log}{\operatorname{Log}}
\newcommand{\Arg}{\operatorname{Arg}}
\newcommand{\Em}{\operatorname{Em}}
\newcommand{\cl}{\operatorname{cl}}

\title{Modelling wave propagation without sampling restrictions using the multiplicative calculus I: Theoretical considerations}

\author{Max Cubillos}

\begin{document}

\date{}

\maketitle

\begin{abstract}
The multiplicative (or geometric) calculus is a non-Newtonian calculus derived
from an arithmetic in which the operations of
addition/subtraction/multiplication are replaced by
multiplication/division/exponentiation.  A major difference between the
multiplicative calculus and the classical additive calculus, and one that has
important consequences in the simulation of wave propagation problems, is that
in geometric calculus the role of polynomials is played by exponentials of a
polynomial argument.  For example, whereas a polynomial of degree one has
constant (classical) derivative, it is the exponential function that has
constant derivative in the multiplicative calculus. As we will show, this
implies that even low-order finite quotient approximations---the analogues of
finite differences in the multiplicative calculus---produce exact
multiplicative derivatives of exponential functions.  We exploit this fact to
show that some partial differential equations (PDE) can be solved far more
efficiently using techniques based on the multiplicative calculus. For wave
propagation models in particular, we will show that it is possible to
circumvent the minimum-points-per-wavelength sampling constraints of classical
methods. In this first part we develop the theoretical framework for studying
multiplicative partial differential equations and their connections with
classical models.
\end{abstract}

\section{Introduction} \label{sec:Introduction}

The calculus developed by Newton and Leibniz is one of most significant
breakthroughs in mathematics but an infinite number of other versions of
calculus are possible.  The treatise~\cite{grossman_non-newtonian_1972} by
Grossman and Katz is perhaps the earliest comprehensive work on other so-called
non-Newtonian calculi where they showed that any pair of injective functions
$\alpha$ and $\beta$ on the real numbers $\R$ can be used to generalize the
arithmetic that is used to define the classical calculus operations of
differentiation and integration.  As a specific example (and perhaps the most
well-known and important example), choosing $\alpha$ to be the identity and
$\beta$ to be the exponential function generates what they called the geometric
calculus, also known as the multiplicative calculus or $^*$-calculus. In this
paper, we focus exclusively on this non-Newtonian calculus (and we will
generally prefer the latter nomenclature).  The $^*$ symbol will typically
refer to multiplicative properties henceforth.

Recent contributions have expanded on the ideas of non-Newtonian calculi and
have shown some applications, particularly using the multiplicative calculus.
These include significant extensions of the multiplicative calculus to complex
numbers~\cite{bashirov_complex_2011,bashirov_riemann_2016}, contributions on
numerical algorithms in the multiplicative
calculus~\cite{aniszewska_multiplicative_2007,misirli_multiplicative_2011,ozyapici_finite_2016,ozyapici_multiplicative_2014,riza_multiplicative_2009}
and applications to specific problems of scientific
interest~\cite{florack_multiplicative_2012,uzer_exact_2015}. However, to the
authors' knowledge there have not been any numerical applications to the partial
differential equation (PDE) of mathematical physics. This paper is the first in
a series of articles that aims to bridge that gap, by applying techniques of
the multiplicative calculus to solve problems in mathematical physics far more
efficiently than current methods. Specifically, we will show that certain wave
problems can be solved with a fixed discretization to an error tolerance that
is independent of the underlying frequency of the solution.

Although the application is novel, the multiplicative calculus certainly is not
and some of the development presented in this paper has been discussed in some
of the previously cited references. For simplicity and completeness, we
rederive multiplicative calculus properties as needed.  However, unlike
previous contributions, we emphasize the distinction between the spaces on
which the classical and multiplicative calculi are most natural and explore the
interactions between these spaces using the natural isomorphisms between them
as well as appropriately defined projection and embedding operators. 

This article is organized as follows: Section~\ref{sec:Review} gives a brief
review of the multiplicative derivative and its properties.
Section~\ref{sec:AdvectionEqn} provides the motivating example of this paper,
where the advection equation is converted to the multiplicative calculus, and
the numerical solution is shown to \emph{exactly} coincide with the exact
solution for any discretization and frequency. In Section~\ref{sec:Formulation}
we then formalize the framework for numerical multiplicative PDEs by
introducing the Riemann surface on which the multiplicative calculus is most
natural (Section~\ref{sec:RiemannSurface}), reviewing the multiplicative
calculus in this space (Section~\ref{sec:ProductCalculus}), and introducing
vector spaces over this field (Section~\ref{sec:VectorSpaces}), proving several
fundamental results in the process. In Section~\ref{sec:Applications}, we apply
the results in this paper to the motivating advection equation example.
Finally, Section~\ref{sec:Conclusion} provides some brief concluding remarks.

\section{A brief review of the multiplicative calculus} \label{sec:Review}

The classical Newtonian derivative of a function $f$ of a real variable $x$ is
given by the limit
\begin{equation}
  f'(x) = \lim_{h \rightarrow 0} \frac{f(x+h)-f(x)}h. 
\end{equation}
The multiplicative derivative, or *derivative, can be derived for positive
functions $f$ by considering a similar limit, with the subtraction in the
numerator replaced by division and the division by $h$ replaced by raising to
the power $1/h$, i.e.,
\begin{equation} \label{eq:LimitDefinition}
  f^*(x) = \lim_{h \rightarrow 0} \left( \frac{f(x+h)}{f(x)} \right)^{\frac 1h}.
\end{equation}
The *derivative may also be written in terms of the classical derivative:
\begin{align}
  f^*(x) &= \lim_{h \rightarrow 0} \left( \frac{f(x+h)}{f(x)} \right)^{\frac 1h} \nonumber \\
         &= \lim_{h \rightarrow 0} e^{\left( \frac{\ln f(x+h) - \ln f(x)}h \right)} \nonumber \\
         &= e^{(\ln f)'(x)} \\ 
         &= e^{\frac{f'(x)}{f(x)}}. 
\end{align}
All the properties of the classical derivative have *derivative counterparts.
For example, if $f$ and $g$ are *differentiable, $h$ is differentiable, and $c$
is a constant, we have~\cite{bashirov_multiplicative_2008}
\begin{enumerate}[itemsep=0pt, label=\roman*.]
\item $(f^c)^* = (f^*)^c$,
\item $(fg)^* = f^*g^*$,
\item $(f/g)^* = f^*/g^*$,
\item $(f^h)^* = (f^*)^h f^{h'}$, \label{item:StarProductRule}
\item $(f \circ h)^* = (f^* \circ h)^{h'}$. \label{item:StarChainRule}
\end{enumerate}

One interpretation of the classical derivative $f'(x_0)$ is that it is the
slope of the line that is tangent to the function at the point $x_0$. By
Taylor's theorem, for a sufficiently smooth function the derivative can be used
to produce the approximation
$$ f(x) = f(x_0) + f'(x_0)(x - x_0) + O((x-x_0)^2). $$
For $f$ equal to a straight line, the remainder term above is zero and the
approximation coincides with the function itself. In the multiplicative
calculus, the analogous functions are $f(x) = a^x$ for a positive number $a$,
for which we have
\begin{equation}
  (a^x)^* = \left( e^{\ln a x} \right)^* = e^{(\ln a x)'} = a.
\end{equation}
In other words, the *derivative of an exponential function is a constant, equal
to the base. Thus, one interpretation of the *derivative of $f$ is that it is
the base of the exponential function tangent to $f$. It is the local
\emph{geometric} rate of change of a function. The multiplicative version of
Taylor's theorem~\cite{bashirov_multiplicative_2008} makes this precise:
$$ f(x) = f(x_0) f^*(x_0)^{(x-x_0)} e^{O((x-x_0)^2)}. $$
It follows that the multiplicative calculus is best suited for describing
exponential type functions. In the next section, we will show how this can be
useful for solving PDEs.

\section{A basic example: the advection equation} \label{sec:AdvectionEqn}

The advantages of the multiplicative calculus over the usual Newtonian calculus
can be seen by considering the advection equation on the positive real line for
a function $u(x,t)$,
\begin{subequations} \label{eq:AdvectionAdditive}
\begin{align}
  u_t + c\,u_x &= 0  \quad x \in \R,\, t > 0, \\
  u(x,0) &= f(x)
\end{align}
\end{subequations}
where $c$ is a real constant and $f$ is some prescribed function. The solution
of equation~\eqref{eq:AdvectionAdditive} is simply
$$ u(x,t) = f(x-ct). $$

We consider an initial condition of the form
\begin{equation} \label{eq:AdvectionIC}
  f(x) = e^{-ax^2} \cos(kx), 
\end{equation}
for some constants $a$ and $k$. This problem embodies one of the fundamental 
challenges in modelling wave physics: For large $k$, the solution is highly
oscillatory and requires a fine mesh to resolve. Low-order methods will
introduce numerical dissipation and/or dispersion, making accurate long time
propoagation difficult.

We now demonstrate how equation~\eqref{eq:AdvectionAdditive} can be efficiently
solved using the multiplicative calculus. But instead of the inital
condition~\eqref{eq:AdvectionIC}, we observe that $f$ is the real part of the
complex function
\begin{equation} \label{eq:AdvectionICComplex}
  g(x) = \exp\left(-\frac{x^2}{2\sigma^2}+i\,kx\right).
\end{equation}
Therefore, we wish to consider the advection equation for complex-valued
$u$. To do so, we first introduce notation for partial *derivatives. For a
function of more than one variable, say $f(x,t)$, we will denote the partial
*derivatives with respect to each variable as $f_x^*$ and $f_t^*$ --- that
is,
$$ f_x^* = \exp \left( \frac 1f \frac{\partial f}{\partial x} \right), \quad
   f_t^* = \exp \left( \frac 1f \frac{\partial f}{\partial t} \right). $$
Higher order partial *derivatives will be denoted similarly; for example, the
second *derivative with respect to $x$ is $f^*_{xx}$.

Using the definition of the partial *derivatives, we divide
equation~\eqref{eq:AdvectionAdditive} by $u$ and exponentiate to obtain
\begin{subequations} \label{eq:AdvectionMultiplicative}
\begin{align}
  u^*_{t}(u^*_{x})^c &= 1, \quad x \in \R, t > 0, \\
  u(x,0) &= g(x).
\end{align}
\end{subequations}
To solve equation~\eqref{eq:AdvectionMultiplicative} numerically we discretize
space and time with the uniform mesh $x_j = j\Delta x$ and $t_n = n\Delta t$,
where $j$ and $n$ are integers. Let the grid function $v^n_j$ be the numerical
solution of equation~\eqref{eq:AdvectionMultiplicative} at the point
$(x_j,t_n)$. We approximate the *derivatives at the point $(x_j,t_n)$ using the
centered finite quotients
\begin{align}
  u^*_{t}(x_j,t_n) &\approx \left( \frac{v^{n+1}_j}{v^{n-1}_j} \right)^{\frac{1}{2\Delta t}}, \\
  u^*_{x}(x_j,t_n) &\approx \left( \frac{v^n_{j+1}}{v^n_{j-1}} \right)^{\frac{1}{2\Delta x}}.
\end{align}
As with classical two-point finite difference schemes, this is an approximation
of the limit definition~\eqref{eq:LimitDefinition}. Substituting these
approximations into equation~\eqref{eq:AdvectionMultiplicative} leads to the
formula
\begin{equation} \label{eq:AdvectionMultDiscrete}
  v^{n+1}_j = v^{n-1}_j\left( \frac{ v^n_{j+1} }{ v^n_{j-1} } \right)^{\frac{ -c\Delta t }{ \Delta x }}.
\end{equation}
We compute the solution for the first time step explicitly. Because this is a
two-step method, we require a second set of initial conditions $v^{-1}_j$,
which we set to be the exact solution, i.e.,
$$ v^{-1}_j = g(j\Delta x + c\Delta t). $$
The solution for $n=1$ of equation~\eqref{eq:AdvectionMultDiscrete} is then
\begin{align}
  v^{1}_j &= e^{ -a(j\Delta x + c\Delta t)^2 } e^{ ik(j\Delta x + c\Delta t) }
            \left( \frac{ e^{ -a(j+1)^2\Delta x^2 } e^{ ik(j+1)\Delta x } }{ e^{ -a(j-1)^2\Delta x^2 } e^{ ik(j-1)\Delta x } } \right)^{\frac{-c\Delta t}{\Delta x}} \label{eq:PreLog} \\
          &= e^{ -a(j^2\Delta x^2 + 2cj\Delta t\Delta x + c^2 \Delta t^2 } e^{ ik(j\Delta x + c\Delta t) }
            \left( e^{ -2acj\Delta x\Delta t } e^{ -2ick\Delta t } \right) \label{eq:PostLog} \\
          &= e^{ -a(j\Delta x - c \Delta t)^2 } e^{ ik(j\Delta x - c\Delta t) } \\
          &= g(j\Delta x - c\Delta t),
\end{align}
which is the exact solution at $x = j\Delta x$ and time $t=\Delta t$. Using an
inductive argument, it is clear that the
scheme~\eqref{eq:AdvectionMultDiscrete} will recover the exact solution to the
problem with initial condition~\eqref{eq:AdvectionICComplex} for all time
\emph{regardless of discretization, wave speed $c$, and frequency $k$}. A few
remarks are in order for this surprising result.

\begin{remark} \label{rem:SecondOrder}
Taking the logarithm of both the continuous and discrete multiplicative
advection equations shows that the scheme is equivalent to a second-order
finite difference solution of the additive advection equation in the unknown
$\log(u)$. The logarithm of our solution is a second-order polynomial
in $x$ and $t$, which explains why our second-order method will produce the exact
solution. This observation is discussed further in Section~\ref{sec:Projection}.
\end{remark}

\begin{remark} \label{rem:Cheat}
The reader may have noticed that we ``cheated'' in going from
equation~\eqref{eq:PreLog} to~\eqref{eq:PostLog}. Unless $\frac{-c\Delta
t}{\Delta x}$ is an integer, evaluating the exponent requires use of the
complex logarithm. Instead, we are implicitly treating the complex numbers as
points on the Riemann surface on which the complex logarithm is single valued.
We have glossed over the issue in this example for simplicity, leaving the
details to Section~\ref{sec:RiemannSurface}.
\end{remark}

\begin{remark}
Although the exact solution in this example is recovered for any $\Delta x$ and
$\Delta t$, this is only in exact arithmetic --- i.e., the scheme is not
unconditionally stable in the presence of round-off errors or with more
general initial conditions. As in Remark~\ref{rem:SecondOrder}, this is clear
by taking the logarithm of equation~\eqref{eq:AdvectionMultDiscrete} and
recognizing that it is just the leapfrog scheme for the additive advection
equation, which is conditionally stable.
\end{remark}

\begin{remark}
Suppose we wish to solve the above problem with the initial condition
\begin{equation}
  f(x) = 1 + \frac 12 e^{-ax^2}\cos(kx).
\end{equation}
One approach to use the multiplicative calculus would be to write $f(x)$ in the form
$f(x) = r(x)e^{i\theta(x)}$ and use the same techniques from this section.
Unfortunately, nothing is gained by doing that: $r(x)$ in this case would be an
oscillating function of $x$ with frequency $k$, and the accuracy of the scheme
above would be comparable to conventional finite differences for the additive
advection equation; e.g., one would have to use an appropriate number of
points per wavelength for an accurate solution. 

This illustrates two points: First, summation of even seemingly simple
functions can pose challenges to the multiplicative calculus and spoil what would
otherwise be a very efficient method. Second, a suitable multiplicative calculus
formulation is problem dependent. In the present example, the ``correct''
approach would be to subtract 1 from $f$ and proceed as before, then add 1 to
the solution at the end. However, it may be difficult (or impossible) to find a
simple multiplicative calculus version of more complicated problems.
\end{remark}

\section{The multiplicative formulation for PDEs} \label{sec:Formulation}

\subsection{The Riemann surface $e^\C$} \label{sec:RiemannSurface}

The key insight in the example of the previous section was to view the function
$\cos(kx)$ as the real part of $e^{ikx}$ and recognize that the exponential
function is the multiplicative calculus analog of the linear function in
classical calculus. However, as noted in Remark~\ref{rem:Cheat}, it is possible
to spoil the simplicity when raising a complex number to a non-integer power,
for example.  The standard technique to evaluate $z^w$ as a single-valued
function is to choose a branch of the complex logarithm, say, the pricipal
branch $\Log z = \ln|z| + i\mathrm{Arg} z$, $\mathrm{Arg} z \in (-\pi,\pi]$,
and then define $z^w \equiv e^{w\Log z}$.  The problem is that
$\mathrm{Arg}(e^{ikx})$ is not a continuous function of $x$: There is a
discontinuity wherever $kx = (2m+1)\pi$ for any integer $m$. In particular,
care must be taken to numerically evaluate the derivative. For example, a
finite difference approximation of the derivative of $\mathrm{Arg}(e^{ikx})$ at
a point $x$ will require samples of the function in the same neighborhood of
continuity. This is problematic when $k$ is large (requiring a fine mesh to
properly sample each continuous segment) and when $kx$ is close to the ``edge''
of a continuous segment (i.e., when $kx \approx (2n+1)\pi$ for any integer
$n$).

Bashirov and Riza recognized the complications introduced by the logarithm when
they developed complex multiplicative calculus in~\cite{bashirov_complex_2011}.
For example, the multiplicative versions of the product and chain rules involve
raising a function to another function. In that paper, they incorporated the
multivalued complex logarithm into the development of the *calculus. No doubt
motivated in part by the desire for a single valued logarithm,
in~\cite{bashirov_riemann_2016} Bashirov and Norozpour developed complex
multiplicative calculus on the Riemann surface on which log is single valued,
which they called $\mathbb B$. However, whereas their interest was in studying
functions mapping $\mathbb B$ onto itself, applications to problems of physical
interest inherently involve functions of classical real and complex variables.
This is the context in which the development here will proceed.

We begin by defining the Riemann surface $\mathbb B=\{(r,\theta) \,:\, r > 0,\,
\theta \in \R\}$ of polar coordinate pairs on which the complex logarithm is
single-valued. If $z=(r,\theta)$ is in $\mathbb B$, then $\log:\mathbb
B\rightarrow\C$ is defined as
\begin{equation}
\log z = \ln r + i\theta.
\end{equation}
Conversely, the exponential function serves as a mapping from $\C$ to $\mathbb
B$: If $z = a + i b \in \C$, then $\exp:\C \rightarrow \mathbb B$ is defined
as
\begin{equation}
\exp z = (e^a,b).
\end{equation}
In fact, $\exp$ is a bijection from $\C$ to $\mathbb B$ and $\log$ is its
inverse. Therefore, it makes sense to denote the Riemann surface as $\mathbb
B=e^\C$ which is the notation we will prefer henceforth, not only to make the
connection to $\C$ more explicit, but it will also easily generalize to the
case of multiplicative vector spaces.

\begin{remark}
Similarly, we use the notation $e^\R$ to refer to the natural domain
of the real multiplicative calculus, which is simply the set of positive real numbers.
Equivalently, it is also the restriction of $e^\C$ to elements with
zero argument.
\end{remark}

\begin{remark} \label{rem:Notation}
We will generally write $\exp z$ to denote the $e^\C$-valued exponential
function and $e^z$ to denote the complex-valued one. However, for $w \in
e^C$, instead of writing the polar pair $w = (r,\theta)$ we will sometimes
write $w = r e^{i\theta}$ where it will not cause confusion. The function $\log
z$ will never refer to the multi-valued complex logarithm, but always to the
inverse of $\exp$. The complex logarithm will only be used on a specified
branch, such as the principle branch ($\Log$).
\end{remark}

The operations of multiplication, division, and raising to a complex power all
have natural definitions in $e^\C$ --- i.e., if $z_1=(r_1,\theta_1)$
and $z_2=(r_2,\theta_2)$ are in $e^\C$ and $w=u+iv \in \C$,
then
$$ z_1 z_2 = (r_1 r_2,\, \theta_1 + \theta_2) $$
$$ \frac{z_1}{z_2} = \left(\frac{r_1}{r_2},\, \theta_1 - \theta_2 \right) $$
\begin{align*}
  z_1^w &\equiv \exp( w \log z_1 ) \\
        &= \left(r_1^u e^{-\theta_1 v},\, \theta_1 u + \ln(r_1) v \right).
\end{align*}

Another important mapping from $e^\C$ to $\C$ is the projection operator, which
we denote by $\Pr:e^\C \rightarrow \C$. If $z=(r,\theta) \in e^\C$, then the
projection is defined as
\begin{equation}
\Pr z = r e^{i\theta} = r\cos\theta + ir\sin\theta.
\end{equation}
Although $(r,\theta)$ with $r = 0$ is not in $e^\C$, the projection operator is
still well defined for such numbers. We therefore define the closure of $e^\C$
to be $\cl e^\C = \{ (r,\theta) \,:\, r \geq 0, \, -\infty < \theta < \infty
\}$ and extend the domain of definition of the projection operator to be
$\cl e^\C$.

Using the projection, rasing $z_1 \in e^\C$ to a power $z_2 \in e^\C$ is
defined as
$$ z_1^{z_2} \equiv z_1^{\Pr z_2}. $$
Note that, unlike $\log$, $\Pr$ is neither one-to-one nor onto. Nevertheless,
its importance in applications is critical. This is discussed in the next
section.

Similar to the projection, we can map $\C$ into $e^\C$ via an embedding
operator. We define the principal embedding $\Em:\C \rightarrow \cl e^\C$
using the principal branch of the complex argument function:
\begin{equation}
  \Em z = (|z|,\Arg z).
\end{equation}
(An embedding with respect to any other branch can be defined in the same way
as a branch of $\log$.) Clearly, $\Pr(\Em w) = w$ for all $w \in \C$, but
$\Em(\Pr z) = z$ only for $z \in \cl e^\C $ with argument in the range
$(-\pi,\pi]$.

\subsection{The roles of $\log$ and $\Pr$ in applications} \label{sec:Projection}

In the last section, we introduced the mappings $\log$ and $\Pr$, both of which
are $\C$-valued functions of a $e^\C$-valued variable. To see their importance,
we return to the advection equation example in Section~\ref{sec:AdvectionEqn}.
In that section, what we were actually doing was numerically solving the problem
\begin{subequations} \label{eq:AdvectionExpC}
\begin{align}
  v^*_{t}(v^*_{x})^c &= 1, \quad x \in \R, t > 0, \\
  v(x,0) &= \exp( -ax^2 + ikx ),
\end{align}
\end{subequations}
for the function $v:\R\times[0,\infty) \rightarrow e^\C$ defined on the Riemann
surface $e^\C$. As one would expect, the solution is $v(x,t) = \exp( -a(x-ct)^2
+ ik(x-ct) )$. The real part of the \emph{projection} of this function is the
desired solution to the problem we were originally intending to solve, i.e.,
equation~\eqref{eq:AdvectionAdditive} with the initial condition $f(x) =
e^{-ax^2}\cos(kx)$.  

On the other hand, recognizing that $\exp$ is a bijection, we may make the
substitution $v = \exp u$ into equation~\eqref{eq:AdvectionExpC} and after
taking the $\log$ the result is also the classical advection
equation~\eqref{eq:AdvectionAdditive}, but with the initial condition $f(x) =
-ax^2+ikx$ --- which is \emph{not} the problem we wanted to solve. However, it
\emph{does} explain why we would expect such good results in our numerical
scheme. The function $f(x) = -ax^2+ikx$ is simply a quadratic polynomial, so a
second-order finite difference scheme will be exact. Heuristically, due to the
bijection $\exp$, we may expect that the analogous multiplicative scheme
(`second-order finite quotients') will exactly recover the solution to
equation~\eqref{eq:AdvectionExpC}. These observations lead us to propose the
following
\begin{ruleofthumb}
To solve a classical PDE using the multiplicative calculus, we seek a *PDE with 
a $e^\C$-valued solution $v$ such that
\begin{enumerate}[itemsep=0pt, label=\roman*.]
\item $\Pr v$ is the ``solution'' of the the classical PDE,
\item $\log v$ is a ``simple'' function.
\end{enumerate}
\end{ruleofthumb}

As mentioned earlier, the fact that $\exp$/$\log$ is a pair of bijections
between $\C$ and $e^\C$ leads one to suspect that every result in classical
calculus has a counterpart in the multiplicative calculus, since for every $f$
on $e^\C$, there is a function $g$ on $\C$ such that $f = \exp g$. We will in
fact rigorously verify this suspicion in several theorems. But in the end,
it is the \emph{projection} that is important to us because that is the
solution to the physical problem as originally posed in the classical calculus
--- and as mentioned in the previous section, $\Pr$ is neither one-to-one nor
onto. Therefore, establishing results that relate the properties of a function
$f$ on $e^\C$ to its projection in $\C$ is critical. This is the main
motivation for the theory in this paper. To begin, we first develop the
multiplicative calculus on $e^\C$.

\subsection{The multiplicative calculus on $e^\C$} \label{sec:ProductCalculus}

We begin with the limit definition of the *derivative on $e^\C$ analogous
to the one given in Section~\ref{sec:Review} for positive real functions.
Throughout this section, $\Omega$ will be a simply connected domain in
either $\R$ or $\C$ (all results will apply to both cases unless noted
otherwise).

\begin{definition}
A function $f:\Omega \rightarrow e^\C$ is said to be
\textbf{\emph{*differentiable at a point $p_0$}} if
$$ d(p_0) = \lim_{p \rightarrow p_0} \left( \frac{f(p)}{f(p_0)} \right)^{1/(p-p_0)} $$
exists and is finite. In this case, the \textbf{\emph{*derivative}} of $f$ is given by
that limit: $f^*(p_0) = d(p_0)$
\end{definition}

The multiplicative derivative can also be obtained from the additive derivative
by recalling that $\C$ and $e^\C$ are isomorphic. Thus, there exists $g:\Omega
\rightarrow \C$ such that $f = \exp g$ and $g = \log f$. Using the definition
of the *derivative,
\begin{align}
  f^*(p_0) &= \lim_{p \rightarrow p_0} \left( \frac{\exp g(p)}{\exp g(p_0)} \right)^{1/(p-p_0)} \nonumber \\
         &= \lim_{p \rightarrow p_0} \exp \left( \frac {g(p) - g(p_0)} {p - p_0} \right) \nonumber \\
         &= \exp g'(p_0) \nonumber \\
         &= \exp (\log f)'(p_0).
\end{align}
Equivalently, substituting $f = \exp g$ and taking the $\log$ of both sides yields
\begin{equation}
  g'(p_0) = \log (\exp g)^*(p_0).
\end{equation}

If $f(p)=(r(p),\theta(p))$, it follows that 
$$ f^*(p) = ( r^*(p),\, \theta'(p) ). $$
Recalling the chain rule for $\C$-valued differentiable functions, we also have
that
\begin{align*}
  (\log f)' &= (\ln r + i\theta)' \\
            &= \frac{r'}r + i \theta' \\
            &= \frac{r'e^{i\theta} + r(i \theta')e^{i\theta}}{re^{i\theta}} \\
            &= \frac{(\Pr f)'}{\Pr f}.
\end{align*}
Conversely, if $g:\Omega \rightarrow \C$ is non-zero and differentiable at a
point $p$, and $\operatorname{em}$ is a branch of the embedding operator whose
cut does not contain $p$, then
\begin{align*}
  g'(p) &= g(p)\frac{g'(p)}{g(p)} \\
        &= g(p) \log (\operatorname{em} g)^*(p).
\end{align*}
We therefore have the following two results, the first of which relating the
*derivative of a $e^\C$-valued function to the derivatives of its logarithm and
projection, and the second relating the derivative of a $\C$-valued function to
the *derivatives of its exponentiation and embedding. 

\begin{theorem} \label{thm:ProductDifferentiability}
For any $p \in \Omega$ and function $f:\Omega \rightarrow e^\C$, the
following are equivalent:
\begin{enumerate}[itemsep=0pt, label=\roman*.]
\item $f$ is *differentiable at $p$, \label{item:Diff1}
\item $\log f$ is differentiable at $p$, \label{item:Diff2}
\item $\Pr f$ is differentiable at $p$. \label{item:Diff3}
\end{enumerate}
Furthermore, if any of the above are true, then
$$ f^*(p) = \exp (\log f)'(p) = \exp \left( \frac{(\Pr f)'(p)}{\Pr f(p)} \right). $$
\end{theorem}

\begin{theorem} \label{thm:AdditiveDifferentiability}
For $p \in \Omega$ and any function $g:\Omega \rightarrow \C$, the following
are equivalent:
\begin{enumerate}[itemsep=0pt, label=\roman*.]
\item $g$ is differentiable at $p$,
\item $\exp g$ is *differentiable at $p$.
\end{enumerate}
Furthermore, if any of the above are true, then
$$ g'(p) = \log (\exp g)^*(p). $$
If $g(p) \neq 0$, let $\operatorname{em}:\C \rightarrow e^\C$ be a branch of
the embedding operator such that $p$ does not lie on its branch cut. Then the
above are also equivalent to 
\begin{enumerate}[itemsep=0pt, label=\roman*.]
\setcounter{enumi}{2}
\item $\operatorname{em} g$ is *differentiable at $p$,
\end{enumerate}
and we have
$$ g'(p) = \log (\exp g)^*(p) = g(p) \log (\operatorname{em} g)^*(p) $$
\end{theorem}

\begin{remark}
Although they are not used in this article, there are also multiplicative
Cauchy-Riemann conditions for *differentiability similar to the classical
Cauchy-Riemann conditions for complex functions. We refer the reader
to~\cite{bashirov_complex_2011,bashirov_riemann_2016} for details.
\end{remark}

Items~\ref{item:Diff1} and~\ref{item:Diff2} in the above theorems represent
what we alluded will be a recurring theme throughout this paper, i.e., the
equivalence of classical and multiplicative calculi via the isomorphism $\exp$.
However, what we are truly interested in is not the study of functions on
$e^\C$ in isolation, but their projection onto $\C$, as in the motivating
example of Section~\ref{sec:AdvectionEqn}. Item~\ref{item:Diff3} in each of the
above theorems provide the first results in this regard.

In applications, we will want to solve a *PDE with a solution whose projection
is the solution of the classical PDE in question. A more basic question is
whether there even exists an $e^\C$-valued function whose projection is a
specified $\C$-valued function. This is important, for example, in converting
an initial condition to the multiplicative calculus. We give the answer here
only in the case of analytic functions.

\begin{theorem} \label{thm:Analytic}
Let $f:\Omega \rightarrow \C$ be a non-zero analytic function. Then
there exists a *analytic function $g:\Omega \rightarrow e^\C$ such
that $f = \Pr g$
\end{theorem}
\begin{proof}
First suppose there exists a continuous curve $\Gamma$ from the origin to
infinity such that the image $f(\Omega)$ does not intersect $\Gamma$. Then
clearly we may take any branch of the argument function whose branch cut is
$\Gamma$, $\mathrm{arg}_\Gamma$, and define $g(z) = ( |f(z)|,
\mathrm{arg}_\Gamma f(z) )$.

If no such branch cut curve exists, we partition $\Omega$ into a set of
subdomains $\Omega_i$ such that $f(\Omega_i)$ does not encircle the origin. By
the preceding paragraph, there is a *analytic function on each subdomain.
Because $\Omega$ is simply connected, we may choose one such subdomain and use
*analytic continuation to construct $g$ on all of $\Omega$.
\end{proof}

Analytic continuation ensures the construction in the proof above, but does not
provide details on how it can be implemented in practice. In many cases, the
simplest way of constructing $g$ is to split $\Omega$ into subdomains
$\Omega_i$ such that $\Arg f$ is onto $(-\pi,\pi]$ in $\Omega_i$
and $\Arg f$ is equal to either $\pm \pi$ on the boundary of the closure of each
$\Omega_i$.  We then define $g_i = \Em f$ on $\Omega_i$. By analytic
continuation, the limits of the $g_i$'s on either side of a boundary between
subdomains have the same modulus and the arguments will differ by an integer
multiple of $2\pi$, since the projection must be continuous. Therefore we
choose one subdomain $\Omega_0$, define $g_0 = \Em f$ on $\Omega_0$, and then
$g_i = g_0 \cdot (1,2\pi m_i)$, where the integers $m_i$ are chosen to ensure
continuity across subdomain boundaries. The following example illustrates the
procedure.

\begin{example}
We consider the Hankel functions $H_n^{(1)}(x)$ and $H_n^{(2)}(x)$ of a
positive real non-zero variable $x$. They are complex-valued functions defined
in terms of the real-valued Bessel funcitons of the first kind, $J_n(x)$, and
second kind, $Y_n(x)$: $H_n^{(1)}(x) = J_n(x) + i Y_n(x)$, $H_n^{(2)}(x) =
J_n(x) - i Y_n(x)$. As functions of $r = \sqrt{x^2+y^2}$ in two dimensions, the
first and second kind Hankel functions represent outgoing and incoming radial
waves, respectively, when multiplied by a factor of the form $e^{-i\omega t}$.

We wish to construct $e^\C$-valued *differentiable functions $G_n^{(1,2)}(x) =
( M_n^{(1,2)}(x), A_n^{(1,2)}(x) )$ such that $\Pr G_n^{(1,2)}(x) =
H_n^{(1,2)}(x)$. The modulus is straightforwardly 
\begin{equation}
  M_n^{(1,2)}(x) = \sqrt{(J_n(x))^2 + (Y_n(x))^2}.
\end{equation}
To obtain the argument $A_n^{(1,2)}$, we recall that the zeros of the Bessel
functions interlace, that the first zero of $J_n$ is less than the first zero
of $Y_n$, and that for $x$ close to $0$, $J_n(x) > 0$ and $Y_n(x) < 0$. Letting
$z_j$ for $j \geq 1$ denote the zeros of $Y_n$ in order of proximity to the
origin, it follows that $\Arg H_n^{(1,2)}(x)$ is continuously differentiable
between even numbered zeros.  Furthermore, for every $z_{2m}$ with positive
integer $m$, the limit of $\Arg H_n^{(1)}(x)$ as $x \rightarrow z_{2n}$ from
the left is $+\pi$, whereas from the right it is $-\pi$. Conversely, the limits
of $\Arg H_n^{(2)}(x)$ as $x$ approaches $z_{2m}$ from the left and right are
$-\pi$ and $+\pi$, respectively. Therefore, we define $A_n^{(1,2)}(x)$ to be
\begin{align}
  A_n^{(1)}(x) &= \Arg( H_n{(1)}(x) ) + 2\pi m \\
  A_n^{(2)}(x) &= \Arg( H_n{(1)}(x) ) - 2\pi m, 
\end{align}
where $m = m(x)$ is the smallest integer such that the zero $z_{2m}$ of
$Y_n(x)$ is greater than $x$ ($m=0$ if $x \leq z_2$).
\end{example}

Having established the basics of continuous multiplicative calculus on $e^\C$,
we now turn to numerical methods. A numerical implementation of multiplicative
PDE solvers will require the concept of vector spaces over $e^\C$, which we
develop next.

\subsection{Vector spaces over $e^\R$ and $e^\C$} \label{sec:VectorSpaces}

A vector space $X$ over a field $K$ requires the operations of multiplication
of vectors by scalars in $K$ and addition of two vectors in $X$. The analogue
in the multiplicative algebra requires properly defined vector multiplication and
exponentiation. 

Let $V$ be such a multiplicative vector space with finite dimension $n$. Given a basis
$B=\{b_1,\dots,b_n\}$, an element $u$ of $V$ can be expressed as an array of
$n$ scalars
$$ u = (u_1,\dots,u_n). $$
As in the case of classical vector spaces where a complex vector is the sum of
a purely real vector and a purely imaginary one, we may write $u = (r,\theta)$
for a vector with components $u_j = (r_j,\theta_j)$, where $r =
(r_1,\dots,r_n)$ and $\theta = (\theta_1,\dots,\theta_n)$.

Anticipating the need for multiplication and exponentiation to be well defined,
the scalars $u_i$ are taken to be in $e^K$. Intuitively, as in the additive
case, we may define vector addition $\oplus$ is defined as componentwise
multiplication and scalar vector multiplication $\otimes$ is defined as
componentwise exponentiation --- that is, if $u = (u_1,\dots,u_n)$ and
$v=(v_1,\dots,v_n)$ are in $V$ and $a$ is in $e^K$, then
\begin{subequations} \label{eq:VectorOps}
\begin{align}
u \oplus v  &= uv  = (u_1v_1,\dots,u_nv_n), \\
a \otimes u &= u^{\log a} = (u_1^{\log a},\dots,u_n^{\log a})
\end{align}
\end{subequations}
are also in $V$. Notice that the $\log$ in the exponentiation is necessary
because the exponent should be an element of $K$, not $e^K$. This also makes
the operation symmetric:
$$ u^{\log a} = (u_1^{\log a},\dots,u_n^{\log a}) = (a^{\log u_1},\dots,a^{\log u_n}) \equiv a^{\log u}. $$

The operations~\eqref{eq:VectorOps} imply that $u = (u_1,\dots,u_n)$ can be
expressed in terms of the basis as
$$ u = \prod_{i=1}^n b_i^{\log u_i}. $$
This is the multiplicative sense in which the scalar $u_i$ is the component of
$u$ in the $b_i$ direction. 

The following result easily follows from the properties of multiplication and
exponentiation in $e^K$.

\begin{proposition}
$V$ defined above is a vector space with vector addition and scalar
multiplication given by $\oplus$ and $\otimes$ in
equations~\eqref{eq:VectorOps}.
\end{proposition}

Note that the role of the origin in a multiplicative vector space is played by the
vector whose components is all ones. We will denote this vector simply as $1$.

Suppose we wish to represent the array $u = (u_1,\dots,u_n)$ in a different basis
$B'=\{b_1',\dots,b_n'\}$.  For each $b_j$ in $B$, there are scalars
$a_1,\dots,a_n \in e^K$ such that $b_j = \prod_i (b_i')^{\log a_{ij}}$. It follows that $u$
in the basis $B'$ is given by
$$ \left( \prod_{i=1}^n u_j^{\log a_{1j}},\dots,\prod_{i=1}^n u_j^{\log a_{nj}} \right) \equiv u^{\log A}, $$
which is the multiplicative analog of matrix-vector multiplication.

An alternative construction of multiplicative vector spaces is to start with an
additive one $X$ with a basis representation $x = (x_1,\dots,x_n)$  and
defining $\exp x$ componentwise, i.e., $\exp x = (\exp x_1,\dots,\exp x_n)$.
For any linear operator $A':X \rightarrow X$, it follows from the change of
basis formula that $\exp(A'x) = (\exp x)^{A'}$. Identifying $A'$ with $\log A$
in the preceding paragraph, it is clear that the definition of $\exp x$ does
not depend on the basis and it is natural to define the space
$$ e^X \equiv \{ v \,:\, v = \exp x,\, x \in X \}. $$
Recalling that scalar-valued $\exp$ is a bijection from $K$ to $e^K$ and $\log$
is its inverse, it follows that componentwise $\log$ is the inverse of the
vector-valued $\exp$ defined above and $X$ is isopmorphic to $e^X$.

Now let $X$ be a vector space endowed with a norm $||\cdot||$.  Then $\exp$ and
$\log$ naturally induce a norm $||\cdot||_*$ in $e^X$:
\begin{equation} \label{eq:StarNorm}
||u||_* = \exp || \log u ||.
\end{equation}
Clearly, $||\cdot||_*$ is a properly defined norm in the multiplicative sense
if and only if $||\cdot||$ is a properly defined norm in the additive sense ---
i.e., for $u,\,v \in e^X$ and $\alpha \in K$, if $||\cdot||$ is a norm we have
\begin{enumerate}[itemsep=0pt, label=\roman*.]
\item $||u||_* \geq 1$ and $||u||_* = 1$ if and only if $u = 1$,
\item $||u^\alpha||_* = ||u||_*^{|\alpha|}$,
\item $||uv||_* \leq ||u||_*||v||_*$.
\end{enumerate}

\begin{remark}
Similarly, any inner product $(\cdot,\cdot)$ in the Euclidean space $K^n$
induces the inner product
$$ (u,v)_* \equiv \exp (\log u,\log v) $$
in the multiplicative Euclidean space $(e^K)^n = e^{nK}$.
\end{remark}

Typically, two normed spaces $X$ and $Y$ are said to be isometrically
isomorphic if there is an isomorphism $\phi:X\rightarrow Y$ such that the norms
of $x$ and $\phi(x)$ are equal for every $x$. However, we do not expect
the same definition to hold for an isomorphism between an additive normed space
and a multiplicative one because of the fundamental difference in the
underlying fields $K$ and $e^K$. (For example, the group identity is $0$
in the former and $1$ in the latter.) Therefore, we extend the definition of
isometric isomorphism to account for this case.
\begin{definition} \label{def:Isomorphic}
Let $X$ and $Y$ be normed spaces over the respective fields $K$ and $L$ with
respective norms $||\cdot||_X:X \rightarrow K$ and $||\cdot||_Y:Y \rightarrow
L$, where the arithmetic in $L$ can be generated from the arithmetic in $K$ by
the function $\alpha:K \rightarrow L$. Then $X$ and $Y$ are
\textbf{\emph{isometrically isomorphic}} if there exists an isomorphism $\phi:X
\rightarrow Y$ such that $\alpha(||x||_X) = ||\phi(x)||_Y $ for every $x \in
X$.
\end{definition}
This leads to the following fundamental result for multiplicative vector
spaces.

\begin{theorem} \label{thm:Isomorphic}
The space $X$ with norm $||\cdot||$ is isometrically isomorphic to the normed
space $e^X$ under the mapping $\exp$ ($\phi = \exp:e^X \rightarrow X$ and
$\alpha = \exp:K \rightarrow e^K$ in Definition~\ref{def:Isomorphic}) with the
norm in $e^X$ given by~\eqref{eq:StarNorm} --- i.e., $\exp ||x|| = ||\exp
x||_*$ for each $x\in X$ (equivalently, $\log ||v||_* = ||\log v||$ for each $v
\in e^X$).
\end{theorem}

Theorem~\ref{thm:Isomorphic} allows all the results and techniques of
finite-dimensional additive vector spaces to be extended directly to
multiplicative vector spaces. In effect, every theorem and method in numerical
linear algebra and numerical differential equations has its multiplicative
counterpart.  Again, as in the case of Theorem~\ref{thm:Analytic},
Theorem~\ref{thm:Isomorphic} implies that nothing new is to be gained from
multiplicative spaces in isolation, since they are essentially additive vector
spaces in disguise. We therefore turn our attention to the projection of $e^X$
onto $X$, which we define as componentwise projection --- that is, for $u =
(u_1,\dots,u_n) \in e^X$ the projection $\Pr u \in X$ is given by
$$ \Pr u = (\Pr u_1,\dots,\Pr u_n). $$
As with the scalar projection operator, we may extend the definition to the
closure of $e^X$, $\cl e^X = \{ (u_1,\dots,u_n) \,:\, u_j \in \cl e^\C \}$.

Now consider $u \in e^X$ with a *norm given by $||u||_*$. What can be said
about the norm of its projection in $X$, $||\Pr u||$? The answer depends on the
norm.

\begin{example}
Let $|\cdot|$ be the absolute value in $\R$ and define the *absolute
value for $y \in e^\R$ as 
\begin{equation}
  |y|_* = \exp|\log y| = 
  \begin{cases}
    y, & y \geq 1 \\
    1/y, & y < 1. 
  \end{cases}
\end{equation}
Clearly, $|\Pr y| \leq |y|_*$ and equality holds for every $y \geq 1$.

Now consider the $\infty$-norm for $x=(x_1,\dots,x_n) \in \R^n$ given by
$||x||_\infty = \max_i |x_i|$ and the induced $*\infty$-norm 
$$ ||u||_{*\infty} = \exp \left( \max_i |\log u_i| \right) = \max_i |u_i|_*. $$
It follows directly from the inequality for the *absolute value that
$||\Pr u||_\infty \leq ||u||_{*\infty}.$ for all $u$ in $e^{n\mathbb
R}$.

This result also carries over to the complex case. The *absolute value of $u =
(r,\theta) \in e^\C$ is $|u|_* = \exp |\log u| = \exp \sqrt{ (\ln r)^2
+ \theta^2 }$. Clearly, $|u|_*$ for non-zero $\theta$ is greater than when
$\theta = 0$, which is simply the real case. Moreover, $|\Pr u | = r$
regardless of the value of $\theta$, which is also a reduction to the real
case. It follows that for any $u$ in $e^{n\C}$, we have
$||\Pr u||_\infty \leq ||u||_{*\infty}$. 
\end{example}

\begin{example}
Now consider the spaces $\R^n$ and $e^{n\R}$, $n \geq 2$ under the 2-norm and
*2-norm repsectively. Consider $x = (x_1,\dots,x_n) \in e^{n\R}$ with $x_j = r
\geq 1$ for all $j$. Then $||\Pr x||_2 = \sqrt{n}r$ and
$$||x||_{*2} = \exp( \sqrt{n(\ln r)^2} ) = r^{\sqrt n}.$$
$||x||_{*2}$ clearly grows faster as a function of $r$, so for $r$
sufficiently large $||x||_{*2} > ||\Pr x||_2$. However, for $r = 1$ we have $1
= ||x||_{*2} < ||\Pr x||_2 = \sqrt{n}$.
\end{example}

As the previous example illustrates, we may not be able to say that $||\Pr u||
\leq ||u||_*$ or vice versa for all $u \in e^{X}$ in general. However, recall
that zero is a singularity in multiplicative spaces, which follows from the
property of the *norms that $||u^\alpha||_* = ||u||^{|\alpha|}$ for any $\alpha
\in \C$. Therefore, the *norm is unbounded as $\alpha \rightarrow -\infty$
while on the other hand $||\Pr u^\alpha|| \rightarrow 0$ since each component
tends to zero. It follows that for any norm and *norm (not necessarily
induced), $||\Pr u|| \leq ||u||_*$ if $\Pr u$ is sufficiently close to zero.

In the context of numerical analysis, perhaps the most pertinent information we
may wish to know is what a numerical approximation in a multiplicative space
can tell us about the corresponding approximation when projected into the
additive space. The main result of this section is that a multiplicative
estimate implies a \emph{relative} additive estimate. The intuition comes from
considering two positive real numbers: if $0<a<b$ then $b - a = a\,( b/a - 1
)$. The result for general vector spaces requires the following lemmas.

\begin{lemma} \label{lem:ScalarMixedProduct}
Let $z \in \C$ and $w \in e^\C$. Then $z\Pr w = \Pr(w\Em z)$.
\end{lemma}

\begin{proof}
If $z = 0$ the equality is trivially satisfied. Otherwise, let $w = (r,\theta)$
and $z = p e^{i\alpha}$ with $p > 0$ and $\alpha \in (-\pi,\pi]$. Then
\begin{align*}
  z\Pr w &= rpe^{i(\theta+\alpha)} \\
         &= \Pr( (rp,\theta+\alpha) ) \\
         &= \Pr( w\Em z ).
\end{align*}
\end{proof}

\begin{lemma} \label{lem:Sum}
$\Pr u +\Pr v  = \Pr( u\Em( 1 + \Pr(v/u) ) )$ for all $u$ and $v$ in $e^X$.
\end{lemma}

\begin{proof}
Let the components of $u$ and $v$ be $u_j = (r_j,\theta_j)$ and $v_j =
(s_j,\phi_j)$, respectively, and let $x = \Pr u + \Pr v$ and $y = \Pr(u\Em( 1 +
\Pr(v/u) ) )$. We prove that $x_j = y_j$ for all $j$. Indeed,
\begin{align*}
x_j &= r_j e^{i\theta_j} + s_j e^{i\phi_j} \\
    &= r_j e^{i\theta_j} \left( 1+ \frac{s_j}{r_j} e^{i(\phi_j - \theta_j)} \right).
\end{align*}
On the other hand, by Lemma~\ref{lem:ScalarMixedProduct}, 
\begin{align*}
y_j &= \Pr( u_j \Em( 1 + \Pr(v_j/u_j) ) ) \\
    &= ( 1 + \Pr(v_j/u_j) ) \Pr u_j \\
    &= \left( 1+ \frac{s_j}{r_j} e^{i(\phi_j - \theta_j)} \right) r_j e^{i\theta_j} \\
    &= x_j.
\end{align*}
\end{proof}

\begin{lemma} \label{lem:Product}
$||\Pr(uv)|| \leq C||\Pr u||\,||\Pr v||$ for all $u$ and $v$ in $e^X$, where
the constant $C>0$ is such that $||x||_\infty \leq C||x||$ for all $x \in X$.
\end{lemma}

\begin{proof}
Let $\Pr u = x = (x_1,\dots,x_n)$ and $\Pr v = y = (y_1,\dots,y_n)$. For
scalars in $e^\C$, the projection of a product is equal to the product of
projections. This property clearly extends to vectors componentwise, so that
$||\Pr(uv)|| = ||(x_1y_1,\dots,x_ny_n)||$, which can be viewed as the weighted
norm of $y$ with weights $x$. Letting $M = \max_j |x_j| = ||x||_\infty$ we have
$||\Pr(uv)|| \leq M ||y|| \leq C||x||\,||y||$.
\end{proof}

Note in particular that $C=1$ for any $p$-norm $||x||_p = \left( \sum_j |x_j|^p
\right)^{1/p}$.

\begin{lemma} \label{lem:Difference}
$|| \Pr u  - 1 || \leq C^{-1} (||u||_*^C - 1)$ for all $u \in e^X$, where
the constant $C$ is the same as in Lemma~\ref{lem:Product}.
\end{lemma}

\begin{proof}
Let $u = (u_1,\dots,u_n)$ and $u_k = (r_k, \theta_k)$. Then
$$ (\Pr(u) - 1)_k = e^{\ln r_k + i\theta_k} - 1 = \sum_{m=1}^\infty \frac 1{m!} (\ln r_k + i\theta_k)^m. $$
Let $\psi^m$ denote the vector whose entries are $(\psi^m)_k = (\ln r_k +
i\theta_k)^m$ and $\psi = \psi^1$. Then $\Pr(u)-1 = \sum_{m=1}^\infty \frac
1{m!} \psi^m$ and it follows from the triangle inequality that
$$ ||\Pr(u) - 1|| = \left|\left| \sum_{m=1}^\infty \frac 1{m!} \psi^m \right|\right| \leq \sum_{m=1}^\infty \frac 1{m!} || \psi^m ||. $$
By Lemma~\ref{lem:Product}, $||\psi^m|| \leq C^{m-1}||\psi||^m$ and we have
$$ ||\Pr(u) - 1|| \leq \frac 1C \sum_{m=1}^\infty \frac 1{m!} (C||\psi||)^m = \frac 1C \left( e^{C||\psi||} - 1 \right) = \frac 1C (||u||_*^C - 1). $$
\end{proof}

\begin{theorem} \label{thm:Estimate}
Let $u$ and $v$ be in $e^X$. Then 
$$ \frac{||\Pr u - \Pr v ||}{||\Pr u||} \leq \left|\left|\frac vu \right|\right|_*^C - 1, $$
where the constant $C$ is the same as in Lemma~\ref{lem:Product}.
\end{theorem}

\begin{proof}
Using Lemmas~\ref{lem:ScalarMixedProduct} and~\ref{lem:Product}, 
$$ || \Pr u - \Pr v|| = \left|\left| \Pr \left( u \Em( \Pr(v/u) - 1 ) \right) \right|\right| \leq C||\Pr u|| \left|\left| \Pr(v/u) - 1 \right|\right|, $$
and by Lemma~\ref{lem:Difference}, $|| \Pr(v/u) - 1 || \leq C^{-1} (||v/u||_*^C - 1)$.
Putting them together proves the theorem.
\end{proof}

The following corollary is an immediate consequence of the theorem.

\begin{corollary}
If a sequence in $e^X$ converges in norm, then its projection in $X$ converges
in norm.
\end{corollary}

Note that the converse is not true --- i.e., a non-convergent sequence in $e^X$
may still have a convergent projection. This is clearly the case when the
underlying field is the complex numbers, since the projection operator is not
one-to-one (arguments differing by a multiple of $2\pi$ have the same
projection), but it is also true in the real case. For example, the projection
of the sequence $x_n = 2^{-n} \in e^\R$ clearly converges to 0, but it does not
converge in $e^\R$ with respect to the *absolute value since $\left|
\frac{x_{n+1}}{x_n} \right|_* = |2^{-1}|_* = 2$ for all $n$. In general, a
sequence in $e^X$ whose projection in $X$ converges to the origin must
necessarily diverge in $e^X$.

\section{Applications} \label{sec:Applications}

In this section we revisit the advection equation example presented in
Section~\ref{sec:AdvectionEqn} within the theoretical framework developed in
the previous sections to gain a better insight of the predicted numerical
results, some of which appear to violate the Nyquist-Shannon Sampling Theorem.
First, we explain how a `low-order' finite quotient approximation can produce
results far more accurately than classical methods and without any sampling
restrictions. Next, we justify rigorously multiplicative calculus-based
numerical methods by proving existence and convergence results for a
variable-coefficient multiplicative advection equation.

\subsection{On the Nyquist-Shannon sampling theorem} \label{sec:Nyquist}

The results of Section~\ref{sec:AdvectionEqn} seem to violate classical
assumptions on aliasing and sampling~\cite{shannon_communication_1949}---i.e.,
a wave signal requires at least two samples per wavelength for its
reconstruction without aliasing. However, we will see below that the solutions
obtained in Section~\ref{sec:AdvectionEqn} are in perfect agreement with the
sampling theorem in the multiplicative space.

We begin with the set $\{ x^k \,:\, 0 \leq k \leq n-1 \}$, which is a basis for
the space of polynomials in the real variable $x$ of degree less than $n$.
Given a set of points $x_j$ and values associated with those points $f_j \in
\C$, we wish to construct the interpolating polynomial --- i.e., we seek
coefficients
$a_k$ such that 
$$ \sum_{k=0}^{n-1} a_k x_j^k = f_j $$
for each $j$. If the number of samples is precisely $n$ and the sample points
$x_j$ are distinct, it is well known that the linear independence of the
monomial basis ensures there is a unique solution to the problem.

By Theorem~\ref{thm:Isomorphic}, under the mapping $\exp$, the set $\{ \exp(x^j)
\,:\, 0 \leq j \leq n \}$ is a basis for the set of \emph{exponential
polynomials} of degree less than $n$, which are functions that can be written as
$$ p(x) = \prod_{j=0}^{n-1} (\exp(x^j))^{a_j}, $$
for some constants $a_j \in \C$. The isomorphism between the polynomials and
exponential polynomials ensures that the multiplicative interpolation problem
has a unique solution as well. In other words, $n$ samples taken at distinct
points uniquely determine an exponential polynomial of degree less than $n$.

Take for example the $e^\C$-valued function $\exp(ikx)$ for $x \in \R$, where
$k$ is a constant. The projection $e^{ikx}$ in the Fourier basis for
$\C$-valued functions has bandwidth $k$ --- this is the highest frequency (and
only) mode. However, in the basis of exponential polynomials , the function
$\exp(ikx)$ has bandwidth 1 regardless of $k$; that is, $\exp(ikx)$ is an
exponential polynomial of degree $1$ and by the previous discussion requires
only $2$ points to determine its coefficients exactly.

Similarly, consider the $e^\C$-valued function $\exp(-ax^2)$, where $a$ is
constant. The projection $e^{-ax^2}$ has infinite bandwidth in the Fourier
basis, although the components decay exponentially with frequency. On the other
hand, in the exponential polynomial basis, $\exp(-ax^2)$ has bandwidth 2
regardless of $a$ and requires $3$ points to determine its coefficients exactly.

It comes as no surprise that some functions with narrow bandwidth in the
additive sense may have large bandwidth in the multiplicative sense. For
example, consider $f(x) = 1 - a e^{ix}$ for $|a| < 1$. In the Fourier basis,
this function has bandwidth 1. We construct a $e^\C$-valued function whose
projection is $f$ using the Taylor expansion of $\log$:
\begin{equation}
 \exp( \log(1 - a e^{ix}) ) = \exp\left( -\sum_{n=1}^\infty \frac{a^n e^{inx}}n \right) = \prod_{n=1}^\infty \exp\left( e^{inx} \right)^{-a^n/n}.
\end{equation}
While we have not introduced the multiplicative Fourier basis of functions in
$e^\C$ and it is beyond the scope of the present discussion, it is not hard to
infer that this function has infinite bandwidth in that basis and that the
modes decay slowly.

The conclusion is that one must consider the types of functions inherent in a
given problem and choose the appropriate calculus in which they are most
efficiently approximated. Because the main challenge in high frequency
wave physics is the sampling constraint imposed by a fundamental carrier
wave, we expect that the multiplicative calculus will be well suited to these
problems.

\subsection{The variable coefficient advection equation} \label{sec:VarAdvectionEqn}

We return to the example of the advection
equation~\eqref{eq:AdvectionAdditive}, but this time with a spatially varying
wave speed. In this section, prove basic existence and convergence results for
the multiplicative advection equation.

\begin{theorem} \label{thm:Existence}
Let $c(x)$ be either strictly positive or negative for all $x \in \R$ and
suppose $f(x)$ is the real part of a non-zero complex-valued analytic 
function $g$. Then there exist functions $v:\R\times[0,\infty) \rightarrow
e^\C$ and $h:\R \rightarrow e^\C$ such that $v(x,t)$ is the solution of
\begin{equation} \label{eq:StarAdvection}
v_{*t}v_{*x}^c = 1, \quad v(x,0) = h(x), 
\end{equation} 
and the real part of the projection of $v$, $u = \operatorname{Re} \Pr v$, is
the solution of 
\begin{equation} \label{eq:ClassicalAdvection}
  u_t + c u_x = 0, \quad u(x,0) = f(x).
\end{equation}
\end{theorem}

\begin{proof}
If $f = \operatorname{Re} g$ and $g \neq 0$, then by
Theorem~\ref{thm:Analytic}, there exists a $e^\C$-valued function such that the
projection is $g$.  This function is $h$ in the statement of the theorem above.

By Theorem~\ref{thm:ProductDifferentiability}, $\log v_{*t} = (\log v)_t$ and
$\log v_{*x} = (\log v)_x$, so taking the $\log$ of
equation~\eqref{eq:StarAdvection} results in the equivalent additive form
\begin{equation} \label{eq:LogAdvection}
  (\log v)_t + c(\log v)_x = 0, \quad \log v(x,0) = \log h.
\end{equation}
The existence of the complex-valued solution to the equation above is a basic
result in linear PDE theory.  Again using
Theorem~\ref{thm:ProductDifferentiability}, we make the substitutions $(\log
v)_t = (\Pr v)_t / \Pr v$ and $(\log v)_x = (\Pr v)_x / \Pr v$ to further
transform the equation to
\begin{equation}
  (\Pr v)_t + c(\Pr v)_x = 0, \quad \Pr v(x,0) = \Pr h.
\end{equation}
By definition, $\Pr h = g$, and taking the real part above completes the proof.
\end{proof}

We now prove the convergence of the projection of the centered finite quotient
solution of the multiplicative advection equation to the solution of the
classical advection equation. Let $c$, $f$, $g$, and $h$ be as in the theorem
above. As in Section~\ref{sec:AdvectionEqn}, for a given $\Delta x$ and $\Delta
t$ we define the grid $x_j = j\Delta x$, $t_n = n\Delta t$, $j$ and $n$
integers, and a norm and the associated *norm on the spatial grid, $||\cdot||$
and $||\cdot||_*$ respectively. Let $c_j = c(x_j)$ and define a $e^\C$-valued
grid function $w_j^n$ as a centered quotient approximation to $v$ satisfying
the multiplicative advection equation, with initial conditions given by the
exact solution:
\begin{subequations} \label{eq:CenteredQuotientScheme}
\begin{align}
  & \left( \frac{w_j^{n+1}}{w_j^{n-1}} \right)^{\frac 1{2\Delta t}} + \left( \frac{w_{j+1}^n}{w_{j-1}^n} \right)^{\frac {c_j}{2\Delta x}} = 0, \\
  & w_j^0 = h(x_j) = v(x_j,0), \quad w_j^{-1} = v(x_j,-\Delta t).
\end{align}
\end{subequations}
For definiteness, we say that a sequence of discretizations $(\Delta x, \Delta
t)$ is $T$-compatible if there is a sequence of integers $N = N(\Delta t)$ such
that $T = N\Delta t$ is fixed (in other words, $T/\Delta t$ must be a positive
integer). We have the following theorem:

\begin{theorem}
For each time $T$ and any $T$-compatible sequence of discretizations $(\Delta
x,\Delta t)$ that converge to zero and satisfy $|c(x)\Delta t/\Delta x| < 1$
for all $x$, the scheme~\eqref{eq:CenteredQuotientScheme} is convergent as
$\Delta x \rightarrow 0$ --- that is, letting $\widetilde{v}_j =
v(x_j,T)$ be the grid function given by the solution
of~\eqref{eq:StarAdvection} at time $T$, we have $||\widetilde{v}/w^N||_*
\rightarrow 0$ ($N = T/\Delta t$) as $\Delta x \rightarrow 0$.  Furthermore,
the projection $\Pr w_j^N$ converges to the solution of the classical advection
equation~\eqref{eq:ClassicalAdvection} at time $T$.
\end{theorem}

\begin{proof}
Convergence of the scheme is equivalent to convergence of the leapfrog scheme
for equation~\eqref{eq:LogAdvection}, which is well known provided the
stability condition $|c(x)\Delta t/\Delta x| < 1$ is satisfied (see,
e.g.,~\cite{strikwerda_finite_2004}). Convergence of the projection to the
solution of the advection equation is a consequence of
Theorems~\ref{thm:Estimate} and~\ref{thm:Existence} --- i.e., if
$\widetilde{u}_j = u(x_j,T)$ is the grid function given by the solution of
equation~\eqref{eq:ClassicalAdvection} at time $T$, then
$$ ||\widetilde{u} - \Pr w^N||
  =    ||\widetilde{u} - \Pr\widetilde{v} + \Pr\widetilde{v} - \Pr w^N||
  \leq ||\widetilde{u} - \Pr\widetilde{v}|| + ||\Pr\widetilde{v} - \Pr w^N||. $$
The first term on the right vanishes by Theorem~\ref{thm:Existence} and the
second term goes to zero by the first part of this theorem and
Theorem~\ref{thm:Estimate}.
\end{proof}

\section{Conclusion} \label{sec:Conclusion}

In this paper, we have demonstrated how the multiplicative calculus can be used
to develop numerical methods without minimum sampling restrictions for wave
problems. A motivating example showed how a simple finite quotient numerical
solution of the advection equation for a Gaussian-modulated sinusoidal initial
condition is equal to the exact solution regardless of frequency or
discretization. The proper framework for understanding the numerical methods
was then presented, emphasizing the use of the Riemann surface $e^\C$ on which
the multiplicative calculus is most natural and extending it to finite
dimensional multiplicative vector spaces. Finally, some initial applications
were presented, explaining the efficiency of methods based on the
multiplicative calculus for wave problems as well as showing their convergence
for the variable coefficient advection equation.

\bibliography{MyLibrary}

\end{document}